\documentclass{amsart}

\usepackage{amssymb}
\usepackage{epsfig}
\usepackage{mathdots}
 \theoremstyle{plain}
\newtheorem{theorem}{Theorem}
\newtheorem{corollary}{Corollary}

\newtheorem{lemma}{Lemma}
\newtheorem{proposition}{Proposition}
\newtheorem{example}{Example}
\theoremstyle{definition}
\newtheorem{definition}{Definition}
\theoremstyle{remark}

\numberwithin{equation}{section}

\newcommand{\bT}{\begin{theorem}}
\newcommand{\eT}{\end{theorem}}
\newcommand{\bProp}{\begin{proposition}}
\newcommand{\eProp}{\end{proposition}}
\newcommand{\bE}{\begin{example}}
\newcommand{\eE}{\end{example}}
\newcommand{\bL}{\begin{lemma}}
\newcommand{\eL}{\end{lemma}}
\newcommand{\bP}{\begin{proof}}
\newcommand{\eP}{\end{proof}}
\newcommand{\bC}{\begin{corollary}}
\newcommand{\eC}{\end{corollary}}
\newcommand{\bD}{\begin{definition}}
\newcommand{\eD}{\end{definition}}
\newcommand{\be}{\begin{enumerate}}
\newcommand{\ee}{\end{enumerate}}
\newcommand{\beqa}{\begin{eqnarray*}}
\newcommand{\eeqa}{\end{eqnarray*}}
\newcommand{\beqaa}{\begin{eqnarray}}
\newcommand{\eeqaa}{\end{eqnarray}}
\newcommand{\ba}{\begin{array}}
\newcommand{\ea}{\end{array}}

\setbox0=\hbox{$+$}
\newdimen\plusheight
\plusheight=\ht0
\def\+{\;\lower\plusheight\hbox{$+$}\;}

\setbox0=\hbox{$-$}
\newdimen\minusheight
\minusheight=\ht0
\def\-{\;\lower\minusheight\hbox{$-$}\;}

\setbox0=\hbox{$\cdots$}
\newdimen\cdotsheight
\cdotsheight=\plusheight
\def\cds{\lower\cdotsheight\hbox{$\cdots$}}

\begin{document}



\thispagestyle{plain}




\vspace{5cc}
\begin{center}

{\Large\bf  A NEW SUMMATION FORMULA FOR WP-BAILEY PAIRS
\rule{0mm}{6mm}\renewcommand{\thefootnote}{}
\footnotetext{\scriptsize 2000 Mathematics Subject Classification: 33D15, 11B65\\
Keywords and Phrases: Bailey pairs, WP-Bailey Chains, WP-Bailey pairs, Lambert Series, Basic Hypergeometric Series.
}}

\vspace{1cc}
{\large\it James Mc Laughlin}

\vspace{1cc}
\parbox{24cc}{{\scriptsize\bf

%

Let $(\alpha_n(a,k),\beta_n(a,k))$ be a WP-Bailey pair. Assuming the
limits exist, let
\[
(\alpha_n^*(a),\beta_n^*(a))_{n\geq 1} = \lim_{k \to
1}\left(\alpha_n(a,k),\frac{\beta_n(a,k)}{1-k}\right)_{n\geq 1}
\]
be the \emph{derived} WP-Bailey pair. By considering a particular
limiting case of a transformation due to George Andrews, we derive
new basic hypergeometric summation and transformation formulae involving derived WP-Bailey
pairs.

We then use these formulae to derive new identities for various theta
series/products which are expressible  in terms of certain types of
Lambert series.

}}
\end{center}

\vspace{1cc}


\vspace{1.5cc}
\begin{center}
{\bf 1. Introduction}
\end{center}


In the present paper we describe a new transformation involving WP-Bailey pairs,
 and describe some of the implications of this transformation. In
 particular, we show how various theta functions which have
 representations in terms of certain kinds of Lambert series, may
 also be represented by basic hypergeometric series involving an
 arbitrary derived WP-Bailey pair. Throughout the paper we employ the usual notations {\allowdisplaybreaks
\begin{align*}
           (a;q)_n &:= (1-a)(1-aq)\cdots (1-aq^{n-1}), \\
          (a_1, a_2, \dots, a_j; q)_n &:= (a_1;q)_n (a_2;q)_n \cdots (a_j;q)_n ,\\
           (a;q)_\infty &:= (1-a)(1-aq)(1-aq^2)\cdots, \mbox{ and }\\
          (a_1, a_2, \dots, a_j; q)_\infty &:= (a_1;q)_\infty (a_2;q)_\infty \cdots (a_j;q)_\infty,
\end{align*}}
and we also assume throughout that $|q|<1$.

Before describing the new results, we first discuss some of the background.
A \emph{WP-Bailey pair} is a pair of sequences $(\alpha_{n}(a,k,q)$,
$\beta_{n}(a,k,q))$
satisfying $\alpha_{0}(a,k,q)$ $=\beta_{0}(a,k,q)=1$ and
{\allowdisplaybreaks
\begin{align}\label{WPpair}
\beta_{n}(a,k,q) &= \sum_{j=0}^{n}
\frac{(k/a;q)_{n-j}(k;q)_{n+j}}{(q;q)_{n-j}(aq;q)_{n+j}}\alpha_{j}(a,k,q).
\end{align}
}

Andrews   \cite{A01} described two methods (see \eqref{wpn2} below for more details of one of these methods) of constructing new WP-bailey pairs from existing pairs. This type of construction is termed a \emph{WP-Bailey chain}, since the process may be iterated to produce a chain of WP-Bailey pairs
\[
(\alpha_{n},\,\beta_{n})\to(\alpha_{n}',\,\beta_{n}')\to
(\alpha_{n}'',\,\beta_{n}'') \to \dots
\]
from any initial pair. These two chains together allow a ``tree" of WP-Bailey pairs to be generated from a single initial pair.

The implications of these two branches of the WP-Bailey tree were further investigated by Andrews and Berkovich in \cite{AB02}. Spiridonov
\cite{S02} derived an elliptic generalization of Andrews first
WP-Bailey chain. Four additional branches were added to
the WP-Bailey tree by Warnaar \cite{W03}, two of which had generalizations to the
elliptic level. More recently,  Liu and Ma \cite{LM08} introduced the idea of a general
WP-Bailey chain, and
added one new branch to the WP-Bailey tree. In \cite{MZ09}, the authors added three new WP-Bailey chains. Of course the special case $k=0$ of a WP-Bailey pair  (a \emph{Bailey pair with respect to $a$}) had been studied for some time prior to Andrews' paper \cite{A01}, and indeed Bailey pairs were used by Slater \cite{S51, S52} to produce her famous list of 130 identities of Rogers-Ramanujan type.

Andrews \cite{A01}   second WP-Bailey chain may be described as follows. If
$(\alpha_{n}(a,k),$ $\beta_{n}(a,k))$ satisfy \eqref{WPpair}, then so
does $(\tilde{\alpha}_{n}(a,k),\,\tilde{\beta}_{n}(a,k))$,
\begin{align}\label{wpn2}
\tilde{\alpha}_{n}(a,k)&= \frac{(qa^2/k)_{2n}}{(k)_{2n}}\left
(\frac{k^2}{q a^2} \right)^n\alpha_{n} \left(a, \frac{q a^2}{k}
\right), \\
\tilde{\beta}_{n}(a,k)&=\sum_{j=0}^{n}
\frac{(k^2/qa^2)_{n-j}}{(q)_{n-j}}\left (\frac{k^2}{q a^2}
\right)^j\beta_{j} \left(a, \frac{q a^2}{k} \right). \notag
\end{align}

It is not difficult to show that \eqref{wpn2} implies (see Corollary
1 in \cite{MZ09b}, for example) that if $(\alpha_n(a,k),
\beta_n(a,k))$ satisfy \eqref{WPpair}, then subject to suitable
convergence conditions,
\begin{multline}\label{wpbteq2b}
\sum_{n=0}^{\infty}  \left(\frac{q a^2}{k^2}\right)^n \beta_n(a,k)\\
= \frac{(qa/k,qa^2/k;q)_{\infty}} {(q a,qa^2/k^2;q)_{\infty}}
\sum_{n=0}^{\infty} \frac{(k;q)_{2n}} {(q a^2/k;q)_{2n}}\left (
\frac{q a^2}{k^2}\right)^n \alpha_n(a,k).
\end{multline}

Indeed, any WP-Bailey chain, including Andrews \cite{A01} first chain, those of Warnaar \cite{W03}, Liu and Ma \cite{LM08} and Mc Laughlin and Zimmer \cite{MZ09} will lead to a similar transformation relating WP-Bailey pairs.

Beginning with Andrews first WP-Bailey chain,
 the present author in \cite{McL09a} derived two new types of transformations relating WP-Bailey pairs.
\begin{theorem}\label{t01}
If $(\alpha_n(a,k), \beta_n(a,k))$ is a WP-Bailey pair, then subject
to suitable convergence conditions, {\allowdisplaybreaks
\begin{multline}\label{wpeq8}
\sum_{n=1}^{\infty} \frac{(q\sqrt{k},-q\sqrt{k},z;q)_{n}(q;q)_{n-1}}
{\left(\sqrt{k},-\sqrt{k}, q k,\frac{q k}{z};q\right)_{n}}\left( \frac{q a}{ z }\right )^{n} \beta_n(a,k)\\
- \sum_{n=1}^{\infty}
\frac{\left(q\sqrt{\frac{1}{k}},-q\sqrt{\frac{1}{k}},\frac{1}{z};q\right)_{n}(q;q)_{n-1}}
{\left(\sqrt{\frac{1}{k}},-\sqrt{\frac{1}{k}}, \frac{q }{k},\frac{q
z}{k};q\right)_{n}}\left( \frac{q z}{ a }\right )^{n}
\beta_n\left(\frac{1}{a},\frac{1}{k}\right) -\\
 \sum_{n=1}^{\infty}\frac{(z;q)_{n}(q;q)_{n-1}}{\left(q a ,\frac{q
a}{z};q\right)_n}\left (\frac{q a}{z}\right)^{n}\alpha_n(a,k)
+
\sum_{n=1}^{\infty}\frac{\left(\frac{1}{z};q\right)_{n}(q;q)_{n-1}}{\left(\frac{q}{a}
,\frac{q z}{a};q\right)_n}\left (\frac{q
z}{a}\right)^{n}\alpha_n\left(\frac{1}{a},\frac{1}{k}\right)\\=
\frac{(a-k)\left(1-\frac{1}{z}\right)\left(1-\frac{ak}{z}\right)}
{(1-a)(1-k)\left(1-\frac{a}{z}\right)\left(1-\frac{k}{z}\right)}
+\frac{z}{k}\frac{\left(z,\frac{q}{z},\frac{k}{a},
\frac{qa}{k},\frac{ak}{z},\frac{qz}{ak},q,q;q\right)_{\infty}}
{\left(\frac{z}{k},\frac{qk}{z},\frac{z}{a},
\frac{qa}{z},a,\frac{q}{a},k,\frac{q}{k};q\right)_{\infty}}.
\end{multline}
}
\end{theorem}
\begin{theorem}\label{c3}
If $(\alpha_n(a,k,q), \beta_n(a,k,q))$ is a WP-Bailey pair, then subject
to suitable convergence conditions, {\allowdisplaybreaks
\begin{multline}\label{wpeq2n}
\sum_{n=1}^{\infty} \frac{(1-k q^{2n})(z;q)_{n}(q;q)_{n-1}}
{(1-k)( q k,q k/z;q)_{n}}\left( \frac{q a}{ z }\right )^{n} \beta_n(a,k,q)\\
+  \sum_{n=1}^{\infty} \frac{(1+k q^{2n})(z;q)_{n}(q;q)_{n-1}}
{(1+k)( -q k,-q k/z;q)_{n}}\left( \frac{-q a}{ z }\right )^{n}
\beta_n(-a,-k,q)\\-2 \sum_{n=1}^{\infty} \frac{(1-k^2
q^{4n})(z^2;q^2)_{n}(q^2;q^2)_{n-1}} {(1-k^2)( q^2 k^2,q^2
k^2/z^2;q^2)_{n}}\left( \frac{q^2 a^2}{ z^2 }\right )^{n}
\beta_n(a^2,k^2,q^2)\\=
\sum_{n=1}^{\infty}\frac{(z;q)_{n}(q;q)_{n-1}}{(q a ,q
a/z;q)_n}\left (\frac{q a}{z}\right)^{n}\alpha_n(a,k,q)\\ +
\sum_{n=1}^{\infty}\frac{(z;q)_{n}(q;q)_{n-1}}{(-q a ,-q
a/z;q)_n}\left (\frac{-q a}{z}\right)^{n}\alpha_n(-a,-k,q)\\-2
\sum_{n=1}^{\infty}\frac{(z^2;q^2)_{n}(q^2;q^2)_{n-1}}{(q^2 a^2 ,q^2
a^2/z^2;q^2)_n}\left (\frac{q^2
a^2}{z^2}\right)^{n}\alpha_n(a^2,k^2,q^2).
\end{multline}
}
\end{theorem}

The results in the present paper follow from a certain limiting case of Andrews' second chain at \eqref{wpn2} above. The main result of the present paper may be described as follows.
Define, for a WP-Bailey pair $(\alpha_n(a,k),
\beta_n(a,k))$ and $n \geq 1$,
\begin{align}\label{ab*}
\alpha_n^{*}&=\alpha_n^{*}(a)=\alpha_n^{*}(a,q) :=\lim_{k \to 1} \alpha_n(a,k),\\
\beta_n^{*}&=\beta_n^{*}(a)=\beta_n^{*}(a,q)=\lim_{k \to 1}
\frac{\beta_n(a,k)}{1-k},\notag
\end{align}
assuming the limits exist. For ease of notation we call such a pair
of sequences $(\alpha_n^*,\beta_n^*)_{n\geq 1}$ a \emph{derived} WP-Bailey
pair.  Define
\begin{multline}\label{f2eq}
f_2(a,q):=
\sum_{n=1}^{\infty}a^{2n}q^{n}\beta_n^{*}(a)-\sum_{n=1}^{\infty}a^{2n}q^{n}\beta_n^{*}(-a)\\
+\sum_{n=1}^{\infty}\frac{(q;q)_{2n-1}}{(qa^2;q)_{2n}}a^{2n}
q^{n}\alpha_n^{*}(-a)
-\sum_{n=1}^{\infty}\frac{(q;q)_{2n-1}}{(qa^2;q)_{2n}}a^{2n}
q^{n}\alpha_n^{*}(a).
\end{multline}
\begin{theorem}\label{t1}
Let $(\alpha_n(a,k),\beta_n(a,k))$ be a WP-Bailey pair and $a$ and
$b$ complex numbers such that all of the derived WP-Bailey pairs
$(\alpha_n^*,\beta_n^*)$ below exist. Let $f_2(a,q)$ be as defined
at \eqref{f2eq} and suppose further that each of series $f_2(a,q)$,
$f_2(b,q)$, $f_2(1/a,q)$, $f_2(1/b,q)$ converges. Then
\begin{multline}\label{t1eq}
f_2(a,q)-f_2(b,q)-f_2(1/a,q)+f_2(1/b,q)\\ =
\frac{2(a-b)(1+ab)}{(1-a^2)(1-b^2)}
-2a\frac{(b/a,qa/b,-ab,-q/ab;q)_{\infty}(q^2,q^2;q^2)_{\infty}}
{(a^2,q^2/a^2,b^2,q^2/b^2;q^2)_{\infty}}.
\end{multline}
\end{theorem}
We remark that one reason this result is of interest is that the
right side is independent of the particular WP-Bailey pair
$(\alpha_n(a,k),\beta_n(a,k))$  employed on  the left side. Note also
that, on the left side, the series involving $a$ are completely separate from those involving $b$, while on the right side $a$ and $b$ are inseparable in some of the infinite products.

We apply this identity, and others proved below, to derive new
series-product identities. For example, if $a$ and $b$ are non-zero complex numbers such that
$aq^n$, $bq^n \not = \pm 1$, for $n \in \mathbb{Z}$, and $|q|<\min
\{|a^2|, 1/|a^2|,|b^2|, 1/|b^2|\}$, then
\begin{multline*}
\sum_{n=1}^{\infty} \bigg[ \left( \frac{(1/a;q)_n}{(qa;q)_{n}}
-\frac{(-1/a;q)_n}{(-qa;q)_{n}}\right)\frac{a^{2n}q^n}{1-q^n}\\-
\left( \frac{(a;q)_n}{(q/a;q)_{n}}
-\frac{(-a;q)_n}{(-q/a;q)_{n}}\right)\frac{q^n}{a^{2n}(1-q^n)}
-\left( \frac{(1/b;q)_n}{(qb;q)_{n}}
-\frac{(-1/b;q)_n}{(-qb;q)_{n}}\right)\frac{b^{2n}q^n}{1-q^n}\\
\phantom{asdasdasdasd}+ \left( \frac{(b;q)_n}{(q/b;q)_{n}}
-\frac{(-b;q)_n}{(-q/b;q)_{n}}\right)\frac{q^n}{b^{2n}(1-q^n)}
\bigg]\\
= \frac{2(a-b)(1+ab)}{(1-a^2)(1-b^2)}
-2a\frac{(b/a,qa/b,-ab,-q/ab;q)_{\infty}(q^2,q^2;q^2)_{\infty}}
{(a^2,q^2/a^2,b^2,q^2/b^2;q^2)_{\infty}}.
\end{multline*}
We also show how various identities relating theta functions to Lambert series may be replaced with a general identity involving an arbitrary derived WP-Bailey pair. For example, Ramanujan's identity
\begin{equation*}
q\psi(q^2)\psi(q^6)=\sum_{n=1}^{\infty}\frac{q^{6n-5}}{1-q^{12n-10}}
-\sum_{n=1}^{\infty}\frac{q^{6n-1}}{1-q^{12n-2}},
\end{equation*}
leads to the identity
\begin{equation*}
q\psi(q^2)\psi(q^6)=\frac{1}{2}\left(f_2(q^{-1},q^6)-f_2(q^{-5},q^6)\right),
\end{equation*}
where $f_2(a,q)$ is as defined at \eqref{f2eq} above.
Thus any derived pair $(\alpha_n^*,$
$\beta_n^*
)$ inserted in \eqref{f2eq} will lead to an expression for $q\psi(q^2)\psi(q^6)$ in terms of basic hypergeometric series, and one particular choice (see Corollary \ref{c3r} below for details)  leads to the following identity.
{\allowdisplaybreaks
\begin{multline*}
2q\psi(q^2)\psi(q^6)=\sum_{n=1}^{\infty}\frac{1+q^{12n-1}}{1+1/q}
\frac{(-1/q,-1/q;q^6)_n(q^6;q^6)_{2n-1}(-q)^{5n}}{(q^6,q^6;q^6)_n(q^4;q^6)_{2n}}\\
-\sum_{n=1}^{\infty}\frac{1-q^{12n-1}}{1-1/q}
\frac{(1/q,1/q;q^6)_n(q^6;q^6)_{2n-1}q^{5n}}{(q^6,q^6;q^6)_n(q^4;q^6)_{2n}}\\
+\sum_{n=1}^{\infty}\frac{1-q^{12n-5}}{1-1/q^5}
\frac{(1/q^5,1/q^5;q^6)_n(q^6;q^6)_{2n-1}q^{n}}{(q^6,q^6;q^6)_n(1/q^4;q^6)_{2n}}\\
-\sum_{n=1}^{\infty}\frac{1+q^{12n-5}}{1+1/q^5}
\frac{(-1/q^5,-1/q^5;q^6)_n(q^6;q^6)_{2n-1}(-q)^{n}}{(q^6,q^6;q^6)_n(1/q^4;q^6)_{2n}}.
\end{multline*}
}

\section{Proof of the Main Identities}

The results in the present paper are derived as consequences of
letting $k \to 1$ in \eqref{wpbteq2b}.
Before coming to the proofs, we recall the $q$-Gauss sum
\begin{equation}\label{qgauss}
\sum_{n=0}^{\infty}\frac{(A,B;q)_n}{(C,q;q)_n}\left( \frac{C}{AB}\right)^n = \frac{(C/A,C/B;q)_{\infty}}{(C,C/AB;q)_{\infty}}.
\end{equation}

\begin{lemma}\label{p1}
Let $(\alpha_{n}(a,k),\,\beta_{n}(a,k))$ be a WP-Bailey pair and $(\alpha_n^*,\beta_n^*)$ the derived pair. For $|q|$, $|qa|$, $|q a^2|$ $<1$ and assuming suitable convergence conditions,
\begin{equation}\label{wp2eq1}
\sum_{n=1}^{\infty}a^{2n}q^{n}\beta_n^{*}(a)
-\sum_{n=1}^{\infty}\frac{(q;q)_{2n-1}}{(qa^2;q)_{2n}}a^{2n} q^{n}\alpha_n^{*}(a)
= f_{1}(a,q),
\end{equation}
where
{\allowdisplaybreaks\begin{align}\label{f1eq}
f_1(a,q)&=\sum_{n=1}^{\infty}\frac{(1/a;q)_na^{2n}q^n}{(aq;q)_n(1-q^n)}\\
&=-\sum_{n=1}^{\infty}\frac{(1-a q^{2n})(a,a;q)_n(q;q)_{2n-1} a^n q^n }{(1-a)(q,q;q)_n(qa^2;q)_{2n}} \notag\\
&=\sum_{n=1}^{\infty}\frac{a^2q^n}{1-a^2q^n} - \sum_{n=1}^{\infty}\frac{aq^n}{1-aq^n}.\notag
\end{align}
}
\end{lemma}
For later use we note that
\begin{equation}\label{f1f2eq}
f_2(a,q)=f_1(a,q)-f_1(-a,q)=-\sum_{n=1}^{\infty}\frac{2a q^n}{1-a^2
q^{2n}},
\end{equation}
where $f_2(a,q)$ is as defined at \eqref{f2eq}.

\begin{proof}[Proof of Lemma \ref{p1}]
Rewrite \eqref{wpbteq2b} as
{\allowdisplaybreaks\begin{multline}\label{wpbteq2bb}
\sum_{n=1}^{\infty}  \left(\frac{q a^2}{k^2}\right)^n \frac{\beta_n(a,k)}{1-k}-\frac{\left(\frac{qa}{k},\frac{qa^2}{k};q\right)_{\infty}} {(q a,qa^2/k^2;q)_{\infty}}
\sum_{n=1}^{\infty} \frac{(kq;q)_{2n-1}} {(q a^2/k;q)_{2n}}\left (
\frac{q a^2}{k^2}\right)^n \alpha_n(a,k)\\
=\frac{1}{1-k}\left( \frac{(qa/k,qa^2/k;q)_{\infty}} {(q a,qa^2/k^2;q)_{\infty}}
-1\right).
\end{multline}}
The left side of \eqref{wp2eq1} now follows upon letting $k \to 1$. To get the first expression for $f_1(a,q)$, use \eqref{qgauss} to expand the infinite product on the right side as an infinite series (set $A=k$, $B=k/a$ and $C=qa$) and then once again let $k \to 1$.

The second expression for $f_1(a,q)$ follows upon substituting the unit WP-Bailey pair
\begin{align}\label{up}
\alpha_{n}(a,k)&=\frac{(q \sqrt{a}, -q
\sqrt{a},a,a/k;q)_n}{(\sqrt{a},-\sqrt{a},q,kq;q)_n}\left(\frac{k}{a}\right)^n,\\
\beta_n(a,k)&=\begin{cases} 1&n=0, \notag\\
0, &n>0,
\end{cases}
\end{align}
into \eqref{ab*} and then inserting the resulting derived pair
\begin{align}\label{up*}
\alpha_{n}^{*}(a)&=\frac{1-a q^{2n}}{1-a}\frac{(a,a;q)_n}{(q,q;q)_n}\left(\frac{1}{a}\right)^n,\\
\beta_n^{*}(a)&=0\notag
\end{align}
 on the left side of \eqref{wp2eq1}.

For the third representation of $f_1(a,q)$ define
\[
G(k):=\frac{(qa/k,qa^2/k;q)_{\infty}} {(q a,qa^2/k^2;q)_{\infty}}
\]
and then
\[
\lim_{k\to 1}\frac{1}{1-k}\left( \frac{(qa/k,qa^2/k;q)_{\infty}} {(q a,qa^2/k^2;q)_{\infty}}
-1\right)=\lim_{k \to 1}\frac{G(k)-G(1)}{1-k}=-G'(1),
\]
and logarithmic differentiation now easily gives the result.
\end{proof}

Remark: The first expression for $f_1(a,q)$ above also follows from inserting the ``trivial" WP-Bailey pair
\begin{align}\label{tp}
\alpha_{n}(a,k)&=\begin{cases} 1&n=0, \\
0, &n>0,
\end{cases}\\
\beta_n(a,k)&=\frac{(k,k/a;q)_n}{(q,aq;q)_n},\notag
\end{align}
into \eqref{ab*} and then inserting the resulting derived pair
\begin{align}\label{triv*}
\alpha_n^*(a)&=0,\\
\beta_n^*(a)&=\frac{(1/a;q)_n}{(aq;q)_n(1-q^n)}\notag
\end{align}
 on the left side of \eqref{wp2eq1}.

One way of viewing Lemma \ref{p1} is as supplying a large number of
representations of the difference of Lambert Series
\[
\sum_{n=1}^{\infty}\frac{a^2q^n}{1-a^2q^n} -
\sum_{n=1}^{\infty}\frac{aq^n}{1-aq^n}.
\]
Indeed, such a representation arises if any pair
$(\alpha_n^*,\beta_n^*)$ deriving from a WP-Bailey is inserted in
\eqref{wp2eq1}, assuming the limits exist and the resulting series
converge.
 We give two example below. The first arises from Singh's WP-Bailey pair
 \cite{S94}:
\begin{align}\label{singhpr}
\alpha_{n}(a,k)&=\frac{(q \sqrt{a}, -q
\sqrt{a},a,\rho_1,\rho_2,a^2q/k\rho_1\rho_2;q)_n}
{(\sqrt{a},-\sqrt{a},q,a q/\rho_1,a q/\rho_2,k\rho_1\rho_2/a;q)_n}\left(\frac{k}{a}\right)^n,\\
\beta_n(a,k)&=\frac{(k \rho_1/a, k\rho_2/a, k,
aq/\rho_1\rho_2;q)_n}{(a q/\rho_1, a q/\rho_2, k \rho_1
\rho_2/a,q;q)_n}. \notag
\end{align}
This gives the derived pair
{\allowdisplaybreaks\begin{align}\label{singhpr*}
\alpha_{n}^*(a)&=\frac{(q \sqrt{a}, -q
\sqrt{a},a,\rho_1,\rho_2,a^2q/\rho_1\rho_2;q)_n}
{(\sqrt{a},-\sqrt{a},q,a q/\rho_1,a q/\rho_2,\rho_1\rho_2/a;q)_n}\left(\frac{1}{a}\right)^n,\\
\beta_n^*(a)&=\frac{( \rho_1/a, \rho_2/a,  aq/\rho_1\rho_2;q)_n}{(a
q/\rho_1, a q/\rho_2,  \rho_1 \rho_2/a;q)_n(1-q^n)}. \notag
\end{align}}
The parameters $\rho_1$ and $\rho_2$ are free, but for simplicity we
let $\rho_1, \rho_2 \to \infty$  to get the derived WP-Bailey pair
\begin{align}\label{singhpr**}
\alpha_n^*(a)&= \frac{1-a q^{2n}}{1-a}\frac{(a;q)_n}{(q;q)_n}(-1)^n q^{n(n-1)/2},\\
\beta_n^*(a)&= \frac{(-1)^n q^{n(n-1)/2}}{a^n(1-q^n)}.\notag
\end{align}
The WP-Bailey  pair
\begin{align}\label{pr3}
\alpha_{n}(a,k)&=\frac{\left(q \sqrt{a}, -q \sqrt{a},a,a
\sqrt{\frac{q}{k}},-a\sqrt{\frac{q}{k}},\frac{a}{\sqrt{k}},
-\frac{aq}{\sqrt{k}},\frac{k}{a};q\right)_n}
{\left(\sqrt{a},-\sqrt{a},q,\sqrt{kq},-\sqrt{kq},q\sqrt{k},-\sqrt{k},\frac{qa^2}{k};q\right)_n}
\left(\frac{k}{a}\right)^n,\\
\beta_n(a,k)&=\frac{\left(\sqrt{k},\frac{k^2}{a^2};q\right)_n}{(q\sqrt{k},q;q)_n},
\notag
\end{align}
provides the derived  pair
\begin{align}\label{pr3*}
\alpha_{n}^*(a)&=\frac{1-aq^{2n}}{1-a}\frac{\left(a,a,
-aq,1/a;q)_n(a^2 q;q^2\right)_n}
{\left(q,q,-1,qa^2;q)_n(q;q^2\right)_n}
\left(\frac{1}{a}\right)^n,\\
\beta_n^*(a)&=\frac{\left(1/a^2;q\right)_n}{2(q;q)_n(1-q^n)}. \notag
\end{align}

\begin{corollary}
If $|q|$, $|qa|$, $|q a^2|$ $<1$, then
{\allowdisplaybreaks\begin{align}\label{singhprpr3eqs}
\sum_{n=1}^{\infty}&\frac{a^2q^n}{1-a^2q^n} -
\sum_{n=1}^{\infty}\frac{aq^n}{1-aq^n}\\
&=\sum_{n=1}^{\infty}\frac{(-a)^n q^{n(n+1)/2}}{1-q^n}
-\sum_{n=1}^{\infty}\frac{1-a
q^{2n}}{1-a}\frac{(a;q)_n(q;q)_{2n-1}}{(q;q)_n(a^2q;q)_{2n}}(-a^2)^n
q^{n(n+1)/2}\notag\\
&=\frac{1}{2}\sum_{n=1}^{\infty}
\frac{\left(1/a^2;q\right)_na^{2n}q^n}{(q;q)_n(1-q^n)}\notag\\
&\phantom{asdasdasdasdasd}-\sum_{n=1}^{\infty}\frac{1-aq^{2n}}{1-a}\frac{\left(a,a,
-aq,1/a;q\right)_n(q^2;q^2)_{n-1}a^n q^n}
{\left(q,q,-1,qa^2;q\right)_n(a^2q^2;q^2)_{n}}.\notag
\end{align}}
\end{corollary}
\begin{proof}
Insert the derived pairs, respectively, at \eqref{singhpr**} and
\eqref{pr3*} into \eqref{wp2eq1}.
\end{proof}

Before proving Theorem \ref{t1}, we recall the result from Lemma 4 in \cite{McL09a}
(this result was previously given by Andrews, Lewis and Liu in \cite{ALL01}, using a different labeling for the parameters): if
\begin{equation}\label{wpeq5}
f(a,k,z,q)
=
\sum_{n=1}^{\infty}\frac{k q^n}{1-kq^n}+ \sum_{n=1}^{\infty}\frac{ q^n a/z}{1-q^n a/z}-
\sum_{n=1}^{\infty}\frac{a q^n}{1-aq^n}- \sum_{n=1}^{\infty}\frac{ q^n k/z}{1-q^n k/z},
\end{equation}
then
\begin{multline}\label{wpeq7}
f(a,k,z,q)-f\left(\frac{1}{a},\frac{1}{k},\frac{1}{z},q\right)=
\frac{(a-k)(1-1/z)(1-ak/z)}{(1-a)(1-k)(1-a/z)(1-k/z)}\\
+\frac{z}{k}\frac{(z,q/z,k/a,qa/k,ak/z,qz/ak,q,q;q)_{\infty}}
{(z/k,qk/z,z/a,qa/z,a,q/a,k,q/k;q)_{\infty}}.
\end{multline}

\begin{proof}[Proof of Theorem \ref{t1}] Replace $k$ with $b$ and set $z=-1$ in \eqref{wpeq7}, to get (after some simple rearrangements) that
\begin{multline*}
f(a,b,-1,q)-f\left(\frac{1}{a},\frac{1}{b},-1,q\right)\\=
 \frac{2(a-b)(1+ab)}{(1-a^2)(1-b^2)}
-2a\frac{(b/a,qa/b,-ab,-q/ab;q)_{\infty}(q^2,q^2;q^2)_{\infty}}
{(a^2,q^2/a^2,b^2,q^2/b^2;q^2)_{\infty}}.
\end{multline*}
The result now follows, upon noting that \eqref{wpeq5} and \eqref{f1f2eq} imply that
\begin{align*}
f(a,b,-1,q)&= \sum_{n=1}^{\infty}\frac{2 b q^n}{1-b^2q^{2n}}
- \sum_{n=1}^{\infty}\frac{2a q^n}{1-a^2q^{2n}}=f_2(a,q)-f_2(b,q),\\
f(1/a,1/b,-1,q)&=f_2(1/a,q)-f_2(1/b,q).
\end{align*}
\end{proof}

As remarked earlier, any derived WP-Bailey pair
$(\alpha_n^*,\beta_n^*)$ may be used in \eqref{t1eq}, providing the
various series involved converge. We also remark that the identities
in the next two corollaries illustrate the somewhat interesting fact
that the series involving $a$ on the left side of \eqref{t1eq} are
completely separable from those involving $b$, while $a$ and $b$ are
inseparable in some of the products on the right side of
\eqref{t1eq}.

We note that the left side of \eqref{t1eq} contains sixteen
different infinite series, so for space saving reasons we give two
example that uses  relatively simple derived pairs. Upon inserting
the  pair at \eqref{singhpr**} in \eqref{t1eq} and performing some
simple collecting and rearranging of terms, the following identity
results.
\begin{corollary} Let $a$ and $b$ be non-zero complex numbers such that $a^2q^n$, $b^2q^n \not = 1$, for $n \in \mathbb{Z}$.
Then {\allowdisplaybreaks\begin{multline}\label{t1c1eq}
\sum_{n=0}^{\infty}\frac{q^{2n^2+3n+1}}{1-q^{2n+1}}
\bigg[b^{2n+1}-a^{2n+1}+\frac{1}{a^{2n+1}}-\frac{1}{b^{2n+1}}\bigg ]
\\
-\frac{1}{2}\sum_{n=1}^{\infty}\frac{(-1)^n
q^{n(n+1)/2}(q;q)_{2n-1}}{(q;q)_n}  \bigg[ \frac{(1-a
q^{2n})(a;q)_n}{(1-a)(qa^2;q)_{2n}}a^{2n}\\ -\frac{(1+a
q^{2n})(-a;q)_n}{(1+a)(qa^2;q)_{2n}}a^{2n} \\-\frac{(1-b
q^{2n})(b;q)_n}{(1-b)(qb^2;q)_{2n}}b^{2n}
+\frac{(1+b q^{2n})(-b;q)_n}{(1+b)(qb^2;q)_{2n}}b^{2n}\\
-\frac{(1- q^{2n}/a)(1/a;q)_n}{(1-1/a)(q/a^2;q)_{2n}}a^{-2n}
+\frac{(1+ q^{2n}/a)(-1/a;q)_n}{(1+1/a)(q/a^2;q)_{2n}}a^{-2n}\\
+\frac{(1- q^{2n}/b)(1/b;q)_n}{(1-1/b)(q/b^2;q)_{2n}}b^{-2n}
-\frac{(1+ q^{2n}/b)(-1/b;q)_n}{(1+1/b)(q/b^2;q)_{2n}}b^{-2n}
\bigg]\\
= \frac{(a-b)(1+ab)}{(1-a^2)(1-b^2)}
-a\frac{(b/a,qa/b,-ab,-q/ab;q)_{\infty}(q^2,q^2;q^2)_{\infty}}
{(a^2,q^2/a^2,b^2,q^2/b^2;q^2)_{\infty}}.
\end{multline}}
\end{corollary}

\begin{corollary} Let $a$ and $b$ be non-zero complex numbers such that
$aq^n$, $bq^n \not = \pm 1$, for $n \in \mathbb{Z}$, and $|q|<\min
\{|a^2|, 1/|a^2|,|b^2|, 1/|b^2|\}$. Then
{\allowdisplaybreaks\begin{multline}\label{t1c2eq}
\sum_{n=1}^{\infty} \bigg[ \left( \frac{(1/a;q)_n}{(qa;q)_{n}}
-\frac{(-1/a;q)_n}{(-qa;q)_{n}}\right)\frac{a^{2n}q^n}{1-q^n}\\-
\left( \frac{(a;q)_n}{(q/a;q)_{n}}
-\frac{(-a;q)_n}{(-q/a;q)_{n}}\right)\frac{q^n}{a^{2n}(1-q^n)}\\
-\left( \frac{(1/b;q)_n}{(qb;q)_{n}}
-\frac{(-1/b;q)_n}{(-qb;q)_{n}}\right)\frac{b^{2n}q^n}{1-q^n}\\
\phantom{asdasdasdasd}+ \left( \frac{(b;q)_n}{(q/b;q)_{n}}
-\frac{(-b;q)_n}{(-q/b;q)_{n}}\right)\frac{q^n}{b^{2n}(1-q^n)}
\bigg]\\
= \frac{2(a-b)(1+ab)}{(1-a^2)(1-b^2)}
-2a\frac{(b/a,qa/b,-ab,-q/ab;q)_{\infty}(q^2,q^2;q^2)_{\infty}}
{(a^2,q^2/a^2,b^2,q^2/b^2;q^2)_{\infty}}.
\end{multline}}
\end{corollary}

\begin{proof}
Insert the derived pair at \eqref{triv*} into \eqref{t1eq}.
\end{proof}

\section{$q$-series/products that are representable in terms of certain Lambert Series}

Many $q$-series and $q$-products have been represented by Ramanujan
and others in terms of Lambert series of the types encountered
earlier. The various expressions for $f_1(a,q)$ and $f_2(a,q)$
stated previously now permit these $q$-series and $q$-products to
expressed in several ways as basic hypergeometric series, one way
for each derived WP-Bailey pair (or arbitrarily many ways, if a
derived pair contains one or more free parameters). We give several
examples to illustrate the different ways in which this may be
accomplished.

Let
\[
a(q):=\sum_{m,n=-\infty}^{\infty}q^{m^2+mn+n^2}.
\]
Here we are using the notation  for this series employed in
\cite{BB91}, where it was shown that
\[
a^3(q)=b^3(q)+c^3(q),
\]where $b(q)=\sum_{m,n=-\infty}^{\infty}\omega^{m-n}q^{m^2+mn+n^2}$,
$\omega=exp(2\pi i/3)$, and
$c(q)=$\\ $\sum_{m,n=-\infty}^{\infty}q^{(m+1/3)^2+(m+1/3)(n+1/3)+(n+1/3)^2}$.
The series $a(q)$ was also studied by Ramanujan, who showed
(\textbf{Entry 18.2.8} of Ramanujan's Lost Notebook - see \cite[page
402]{AB05}) that
\begin{equation}\label{corlameq2}
a(q)=1+6\sum_{n=1}^{\infty}\frac{q^{-2}q^{3n}}{1-q^{-2}q^{3n}}
-6\sum_{n=1}^{\infty}\frac{q^{-1}q^{3n}}{1-q^{-1}q^{3n}}.
\end{equation}
From this identity we may deduce several representations of $a(q)$
in terms of basic hypergeometric series.
\begin{corollary}\label{corlam1} If $\rho_1, \rho_2 \not = 0$ and
$0<|q|<1$, then {\allowdisplaybreaks
\begin{align}\label{corlameq1}
a(q)&=1+6 \sum_{n=1}^{\infty}\frac{(q;q^3)_n q^n}
{(q^2;q^3)_n(1-q^{3n})},\\
&=1-6\sum_{n=1}^{\infty}
\frac{(1-q^{6n-1})(1/q,1/q;q^3)_n(q^3;q^3)_{2n-1}q^{2n}}
{(1-1/q)(q^3,q^3;q^3)_n(q;q^3)_{2n}},\notag\\
&=1+6\sum_{n=1}^{\infty}
\frac{(\rho_1q,\rho_2q,q^2/\rho_1\rho_2;q^3)_n q^n}
 {(q^2/\rho_1,q^2/\rho_2,\rho_1\rho_2q;q^3)_n(1-q^{3n})}\notag\\
 &\phantom{asasd}-6\sum_{n=1}^{\infty}
\frac{(1-q^{6n-1})(1/q,\rho_1,\rho_2,
q/\rho_1\rho_2;q^3)_n(q^3;q^3)_{2n-1}q^{2n}}
{(1-1/q)(q^2/\rho_1,q^2/\rho_2,\rho_1\rho_2
q,q^3;q^3)_n(q;q^3)_{2n}},\notag\\
&=1+6\sum_{n=1}^{\infty}\frac{(-1)^nq^{(3n^2+n)/2}}{1-q^{3n}}\notag\\
&\phantom{aasdas} -6 \sum_{n=1}^{\infty}
\frac{(1-q^{6n-1})(1/q;q^3)_{n}(q^3;q^3)_{2n-1}(-1)^nq^{(3n^2-n)/2}}
{(1-1/q)(q^3;q^3)_n(q;q^3)_{2n}}.\notag
\end{align}
}
\end{corollary}
\begin{proof}
From \eqref{corlameq2} it can be seen that $a(q)=1+6f_1(1/q,q^3)$,
where $f_1(a,q)$ is as defined at \eqref{f1eq}. The first two
equalities follow from the other two representations of
$1+6f_1(1/q,q^3)$ that derive from the right side of \eqref{f1eq}.
The last two equalities follow from substituting the derived pairs
at \eqref{singhpr*} and \eqref{singhpr**} into $1+6\times$(the left
of \eqref{wp2eq1}) (with $q$ replaced with $q^3$ and $a$ replaced
with $1/q$).
\end{proof}

Recall that
\[
\psi
(q):=\sum_{n=0}^{\infty}q^{n(n+1)/2}=\frac{(q^2;q^2)_{\infty}}{(q;q^2)_{\infty}},
\]
is one of the theta functions studied extensively by Ramanujan.
\begin{corollary}\label{c2}
Let $|\rho_1|>1$ and $i=\sqrt{-1}$. Then
{\allowdisplaybreaks\begin{align}\label{psiq42cor}
&\psi^2(q^4)=\frac{i}{2q}\bigg(\sum_{n=1}^{\infty}
\frac{(1+iq^{4n-1})(-i/q,-i/q;q^2)_n(q^2;q^2)_{2n-1}(-iq)^n}
{(1+i/q)(q^2,q^2;q^2)_n(-1;q^2)_{2n}}\\
&\phantom{aasdaasdasdad}-
\sum_{n=1}^{\infty}
\frac{(1-iq^{4n-1})(i/q,i/q;q^2)_n(q^2;q^2)_{2n-1}(iq)^n}
{(1-i/q)(q^2,q^2;q^2)_n(-1;q^2)_{2n}}\bigg),\notag\\
&=\frac{1}{2iq}\bigg(\sum_{n=1}^{\infty}\frac{(i\rho_1 q;q^2)_n\left(-1\right)^n}{(-i
q/\rho_1;q^2)_n(1-q^{2n})\,\rho_1^n}
-\sum_{n=1}^{\infty}\frac{(-i\rho_1 q;q^2)_n\left(-1\right)^n}{(i
q/\rho_1;q^2)_n(1-q^{2n})\,\rho_1^n}\notag\\
&\phantom{aasdaasdasdad}+ \sum_{n=1}^{\infty}
\frac{(1-iq^{4n-1})(i/q,\rho_1;q^2)_n(q^2;q^2)_{2n-1}(-1)^n}
{(1-i/q)(iq/\rho_1,q^2;q^2)_n(-1;q^2)_{2n}\,\rho_1^n}\notag\\
&\phantom{aasdaasdasdad}-
\sum_{n=1}^{\infty}
\frac{(1+iq^{4n-1})(-i/q,\rho_1;q^2)_n(q^2;q^2)_{2n-1}(-1)^n}
{(1+i/q)(-iq/\rho_1,q^2;q^2)_n(-1;q^2)_{2n}\,\rho_1^n}
\bigg),\notag\\
&=\sum_{n=0}^{\infty}\frac{(-1)^n q^{4n^2+4n}}{1-q^{4n+2}}+\frac{1}{2iq}\bigg(
\sum_{n=1}^{\infty}
\frac{(1-iq^{4n-1})(i/q;q^2)_n(q^2;q^2)_{2n-1}q^{n^2-n}}
{(1-i/q)(q^2;q^2)_n(-1;q^2)_{2n}}\notag\\
&\phantom{aasdaasdasdad} -\sum_{n=1}^{\infty}
\frac{(1+iq^{4n-1})(-i/q;q^2)_n(q^2;q^2)_{2n-1}q^{n^2-n}}
{(1+i/q)(q^2;q^2)_n(-1;q^2)_{2n}}\bigg).\notag
\end{align}}
\end{corollary}
{\allowdisplaybreaks
\begin{proof}
By Example (iv) in Section 17 of Chapter 17 of Ramanujan's second notebook (see \cite[page 139]{B91}),
\[
\psi^2(q^2)=\sum_{n=0}^{\infty}\frac{q^n}{1+q^{2n+1}},
\]
so that {\allowdisplaybreaks\begin{align}\label{psieq2}
\psi^2(q^4)&=\sum_{n=1}^{\infty}\frac{q^{2n-2}}{1+q^{4n-2}}
=\frac{1}{2iq}\sum_{n=1}^{\infty}\frac{-2\left( \frac{1}{iq}\right)q^{2n}}{1-\left(\frac{1}{iq} \right)^2q^{4n}}\\
&=\frac{1}{2iq}f_2\left(\frac{1}{iq},q^2 \right)
=\frac{1}{2iq}\left(f_1\left(\frac{1}{iq},q^2 \right)-  f_1\left(\frac{-1}{iq},q^2 \right)   \right).\notag
\end{align}}
The first equality now follows from \eqref{f1eq}, using the second representations for $f_1(1/iq,q^2)$ and $f_1(-1/iq,q^2)$.

The second equality is a consequence of letting $\rho_2 \to \infty$ in the derived pair at \eqref{singhpr*} and substituting the resulting derived pair into the expression for $f_2(1/iq,q^2)/(2iq)$ that follows from \eqref{f2eq}.

The third equality is a consequence of letting $\rho_1 \to \infty$ in the second equality.
\end{proof}}

For a third example, we recall another
identity of Ramanujan (see \textbf{Entry 3} (i), Chapter 19, page
223 of \cite{B91}):
\begin{equation}\label{rqp2p6}
q\psi(q^2)\psi(q^6)=\sum_{n=1}^{\infty}\frac{q^{6n-5}}{1-q^{12n-10}}
-\sum_{n=1}^{\infty}\frac{q^{6n-1}}{1-q^{12n-2}},
\end{equation}

From \eqref{f1f2eq},
\begin{align*}
q\psi(q^2)\psi(q^6)&=\frac{1}{2}\left(f_2(q^{-1},q^6)-f_2(q^{-5},q^6)\right)\\
&=\frac{1}{2}\left(f_1(q^{-1},q^6)-f_1(-q^{-1},q^6)-f_1(q^{-5},q^6)+f_1(-q^{-5},q^6)\right)
\end{align*}
Thus any derived pair $(\alpha_n^*(a,q), \beta_n^*(a,q))$ inserted in Lemma \ref{p1}, with $q$ replaced with $q^6$ and $a$ taking the values $q^{-1}, -q^{-1}, q^{-5}, -q^{-5}$ will give an expression for $q\psi(q^2)\psi(q^6)$ containing 8 series. However, for simplicity, we use the pair at \eqref{up*} (so $\beta_n^*(a)=0$, reducing the 8 series to 4) to get the following identity.

\begin{corollary}\label{c3r}
{\allowdisplaybreaks\begin{multline}\label{c3eq}
2q\psi(q^2)\psi(q^6)=\sum_{n=1}^{\infty}\frac{1+q^{12n-1}}{1+1/q}
\frac{(-1/q,-1/q;q^6)_n(q^6;q^6)_{2n-1}(-q)^{5n}}{(q^6,q^6;q^6)_n(q^4;q^6)_{2n}}\\
-\sum_{n=1}^{\infty}\frac{1-q^{12n-1}}{1-1/q}
\frac{(1/q,1/q;q^6)_n(q^6;q^6)_{2n-1}q^{5n}}{(q^6,q^6;q^6)_n(q^4;q^6)_{2n}}\\
+\sum_{n=1}^{\infty}\frac{1-q^{12n-5}}{1-1/q^5}
\frac{(1/q^5,1/q^5;q^6)_n(q^6;q^6)_{2n-1}q^{n}}{(q^6,q^6;q^6)_n(1/q^4;q^6)_{2n}}\\
-\sum_{n=1}^{\infty}\frac{1+q^{12n-5}}{1+1/q^5}
\frac{(-1/q^5,-1/q^5;q^6)_n(q^6;q^6)_{2n-1}(-q)^{n}}{(q^6,q^6;q^6)_n(1/q^4;q^6)_{2n}}.
\end{multline}}
\end{corollary}

There are a number of other identities where theta functions are
expressed in terms of certain Lambert series,
 which
may be treated similarly to derive results like those in this
section. These include \textbf{Entry 34 (p.284)} in chapter 36 of
Ramanujan's notebooks (see \cite[page 374]{B98}),
\begin{equation}\label{psieq1}
q\frac{\psi^3(q^3)}{\psi(q)}=\sum_{n=1}^{\infty}\frac{q^{3n-2}}{1-q^{6n-4}}
-
\sum_{n=1}^{\infty}\frac{q^{3n-1}}{1-q^{6n-2}}
=\frac{f_2(1/q,q^3)-f_2(1/q^2,q^3)}{2},
\end{equation}
and others.

\section{Concluding Remarks}

In the present paper and its companion \cite{McL09a} we considered limiting
cases of the two WP-Bailey chains described by Andrews in \cite{A01}.
There are a number of other WP-Bailey chains described in the literature
(see the papers of  Warnaar \cite{W03},  Liu and Ma \cite{LM08}
and Mc Laughlin and Zimmer \cite{MZ09}), and it may be that a similar
analysis of some of these chains may also have interesting consequences.

%
%

\vspace{2cc}

\begin{center}{\small\bf REFERENCES}
\end{center}


\vspace{0.8cc}
\newcounter{ref}
\begin{list}{\small \arabic{ref}.}{\usecounter{ref} \leftmargin 4mm \itemsep -1mm}

\bibitem{A01}
G. E. Andrews, \emph{Bailey's transform, lemma, chains and
tree.} Special functions 2000: current perspective and future
directions (Tempe, AZ), 1--22, NATO Sci. Ser. II Math. Phys. Chem.,
\textbf{30}, Kluwer Acad. Publ., Dordrecht, 2001.

\bibitem{AB02}
G. E. Andrews, A. Berkovich, \emph{The WP-Bailey
tree and its implications.}
 J. London Math. Soc. (2) \textbf{66} (2002), no. 3, 529--549.

\bibitem{AB05}
G.E. Andrews and B.C. Berndt,
``Ramanujan's Lost Notebook, Part I", Springer, 2005.

\bibitem{ALL01}
G. E. Andrews, R. Lewis,Z. G. Liu,
\emph{An identity relating a theta function to a sum of Lambert series.}
Bull. London Math. Soc. \textbf{33} (2001), no. 1, 25--31.



\bibitem{B91}
B. C. Berndt, ``Ramanujan's Notebooks, Part III",
Springer-Verlag, New York, 1991.

\bibitem{B94}
B. C. Berndt, ``Ramanujan's Notebooks, Part IV",
Springer-Verlag, New York, 1994.

\bibitem{B98}
B. C. Berndt, ``Ramanujan's Notebooks, Part V",
Springer-Verlag, New York, 1998.

\bibitem{BB91}
J. M. Borwein,  P. B. Borwein, \emph{A cubic counterpart of
Jacobi's identity and the AGM.} Trans. Amer. Math. Soc. \textbf{323}
(1991), no. 2, 691--701.






\bibitem{LM08}
Q. Liu; X. Ma \emph{On the Characteristic Equation of Well-Poised
Bailey Chains.} Ramanujan J. 18 (2009), no. 3, 351--370.

\bibitem{McL09a}
J. Mc Laughlin,
\emph{Some new Transformations for Bailey pairs and WP-Bailey Pairs.} Cent. Eur. J. Math. \textbf{8} (2010), no. 3, 474–-487.




\bibitem{MZ09}
J. Mc Laughlin, P.  Zimmer,  \emph{General
WP-Bailey Chains.} Ramanujan J. 22 (2010), no. 1, 11--31.

\bibitem{MZ09b}
J. Mc Laughlin, P. Zimmer,  \emph{Some Implications of the WP-Bailey Tree.} Adv. in Appl. Math.  \textbf{43}, no. 2, August 2009, Pages 162-175.


\bibitem{S94}
U. B. Singh, \emph{A note on a transformation of Bailey.}
 Quart. J. Math. Oxford Ser. (2) \textbf{45} (1994), no. 177, 111--116.
\bibitem{S51}

L. J. Slater, \emph{A new proof of Rogers's transformations of
infinite series.} Proc. London Math. Soc. (2) 53, (1951). 460--475.

\bibitem{S52}
L. J. Slater,
 \emph{Further identities of the Rogers-Ramanujan type},
 Proc. London Math.Soc. \textbf{54} (1952) 147--167.

\bibitem{S02}
V. P. Spiridonov,
\emph{An elliptic incarnation of the Bailey chain.} Int. Math. Res.
Not. 2002, no. \textbf{37}, 1945--1977.


\bibitem{W03}
S. O. Warnaar, \emph{Extensions of the well-poised and elliptic
well-poised Bailey lemma.} Indag. Math. (N.S.) \textbf{14} (2003),
no. 3-4, 571--588.

\end{list}

\vspace{1cc}


 West Chester University, 25 University Avenue, West Chester, PA 19383.\\
 jmclaughl@wcupa.edu

{\small
\noindent

}\end{document}